\documentclass[a4paper,11pt,twoside]{amsart}

\usepackage{amsmath}
\usepackage{amsthm}
\usepackage{amsfonts}
\usepackage{amssymb}
\usepackage{mathrsfs}
\usepackage{graphicx}
\usepackage{enumerate}
\usepackage{url}

\usepackage{wasysym}

\usepackage{tikz}
\usepackage{tikz-cd}

\usetikzlibrary{decorations.pathreplacing}

\usepackage{subcaption}

\usetikzlibrary{arrows}
\usetikzlibrary{decorations.markings}

\tikzstyle{CCC}=[shape=circle, draw, fill=black!10, align=center, font=\scriptsize]
\tikzstyle{CC}=[shape=circle, draw, align=center, font=\scriptsize]

\theoremstyle{plain}
\newtheorem{lemma}{Lemma}[section]
\newtheorem{prop}[lemma]{Proposition}
\newtheorem{theo}[lemma]{Theorem}

\theoremstyle{remark}
\newtheorem{rem}[lemma]{Remark}
\newtheorem*{notat}{Notation}
\theoremstyle{definition}
\newtheorem{definition}[lemma]{Definition}
\newtheorem{ex}[lemma]{Example}

\newcommand{\N}{\mathbb{N}}
\newcommand{\Pc}{\mathcal{P}}

\newcommand{\Gm}{\Gamma}
\newcommand{\Gmc}{\Gamma_C}
\newcommand{\op}{\textup{op}}
\newcommand{\K}{\mathcal{K}}
\newcommand{\X}{\mathcal{X}}
\newcommand{\mxl}{\textup{mxl}}
\newcommand{\mnl}{\textup{mnl}}
\newcommand{\Sts}{\textnormal{st}}
\newcommand{\st}[1]{\textnormal{st}(#1)}
\newcommand{\St}{\textnormal{St}}
\newcommand{\I}{\mathcal{I}}
\newcommand{\id}{\textup{id}}
\newcommand{\im}{\textup{Im}}

\begin{document}

\title{The crosscut poset}

\author{Miguel Ottina}
\address{}
\email{miguelottina@gmail.com}

\subjclass[2010]{Primary: 06A06, 06A07. Secondary: 54H25, 55M20.}


\keywords{Crosscut complex; Crosscut poset; Fixed point property; Fixed simplex property.}

\thanks{Research partially supported by grants M044 (2016--2018) and 06/M118 (2019--2021) of Universidad Nacional de Cuyo.}

\begin{abstract}
We introduce a new combinatorial invariant, which we call crosscut poset, that is finer than the crosscut complex. We exhibit many applications of the crosscut poset which include a generalization of Bj\"orner's crosscut theorem and two results concerning the fixed point property and the fixed simplex property.
\end{abstract}

\maketitle

\section{Introduction}

A classical technique to study combinatorial objects such as posets or graphs is to associate a simplicial complex to each of them and then infer properties of them from topological properties of the associated simplicial complex. Well-known examples of this technique are the order complex, the crosscut complex \cite{bjorner1981homotopy}, Lov\'asz's neigbourhood complex \cite{lovasz1978kneser} and $\mathrm{Hom}$--complexes \cite{babson2003topological}.

The crosscut complex of a poset is a simplicial complex constructed from a cutset of the given poset. Specifically, given a poset $P$ and a cutset $X$ of $P$ the crosscut complex of $P$ (with respect to $X$) is the simplicial complex $\Gmc(P,X)$ whose simplices are the non-empty finite subsets of $X$ which are astral (a subset $A\subseteq P$ is astral if there exists $a\in P$ such that every element of $A$ is comparable to $a$).

The notion of crosscut complex was introduced by Bj\"orner in \cite{bjorner1981homotopy} to study homotopical properties of posets. One of the main results of that work is Bj\"orner's Crosscut Theorem \cite[Theorem 2.3]{bjorner1981homotopy}, which states that, under suitable assumptions, the geometric realizations of the crosscut complex and the order complex of a poset are homotopy equivalent. This result was used in \cite{baclawski1979fixed} to study the fixed point property of posets but has many other applications in Combinatorics, Topology and Algebra \cite{bjorner1999complexes,clark2010poset,kozlov2006simple,mapes2013finite,minian2014note,olteanu2016buchberger}. Moreover, in the case of finite posets, a simple homotopy variant of Bj\"orner's Crosscut Theorem is given by Barmak in \cite[Theorem 4.5]{barmak2011quillen}.

In this article we introduce a new combinatorial invariant, which we call \emph{crosscut poset}, which is strongly related to Bj\"orner's crosscut complex. The main idea of the crosscut poset is to consider, for each non-empty astral subset $A$ contained in a given cutset $X$ of a poset $P$, the connected components of the subposet $\st{A}$ of the elements of $P$ which are comparable to all the elements of $A$. These connected components are then ordered by inclusion to obtain the crosscut poset. Note that a subset $A$ is astral if and only if $\st{A}\neq\varnothing$. Therefore, instead of analysing whether $\st{A}$ is non-empty as in the crosscut complex, we track the connected components of the non-empty subposets $\st{A}$. In this way, we are able to obtain much more information of the poset $P$ and, in consequence, the crosscut poset is a finer invariant than the crosscut complex. In addition, we prove that, under the hypotheses of Bj\"orner's Crosscut Theorem, the geometric realization of the crosscut complex is homotopy equivalent to the geometric realization of the order complex of the crosscut poset, which gives an important relationship between the crosscut complex and the crosscut poset.

One of the main results of this article states that, under suitable assumptions, the crosscut poset of a poset $P$ is weak homotopy equivalent to $P$, which, as we shall see, gives a generalization to Bj\"orner's Crosscut Theorem and to Barmak's simple homotopy variant. By means of the crosscut poset we also generalize a result of Baclawski and Bj\"orner's which gives a sufficient condition for a poset to have the strong fixed point property \cite[Proposition 3.1]{baclawski1979fixed} and we obtain two additional results concerning the fixed point property and the fixed simplex property.

\section{Preliminaries and notations}

Let $P$ be a poset. The sets of maximal and minimal elements of $P$ will be denoted by $\mxl(P)$ and $\mnl(P)$ respectively and  the opposite poset of $P$ will be denoted by $P^{\op}$. In addition, for each $a\in P$ we denote
\begin{align*}
P_{\leq a} & = \{x\in P \mid x\leq a\}, & P_{\geq a} & = \{x\in P \mid x\geq a\}, \\
P_{< a} &= \{x\in P \mid x< a\}, & P_{> a} &= \{x\in P \mid x > a\}.
\end{align*}
The order complex of $P$ will be denoted by $\K(P)$ and its geometric realization will be denoted by $|\K(P)|$. On the other hand, if $K$ is a simplicial complex, $\X(K)$ will denote the poset of simplices of $K$ ordered by inclusion.

\begin{definition}
Let $P$ be a poset.
\begin{itemize}
\item Let $a\in P$. We define the \emph{star of $a$ in $P$} as the subposet $\Sts_P(a)=P_{\leq a} \cup P_{\geq a}$.
\item Let $A$ be a non-empty subset of $P$. We define the \emph{star of $A$ in $P$} as the subposet 
\begin{displaymath}
\Sts_P(A)=\bigcap_{a\in A}\Sts_P(a).
\end{displaymath}
\end{itemize}
\end{definition}
When there is no risk of confusion we will denote $\Sts_P(a)$ by $\st{a}$ and $\Sts_P(A)$ by $\st{A}$.

\smallskip

Recall that if $K$ is a simplicial complex, $S_K$ is the set of simplices of $K$ and $\sigma\in S_K$, the \emph{closed star} of $\sigma$ in $K$ is the subcomplex $\St_K(\sigma)$ of $K$ whose set of simplices is $\{\tau\in S_K \mid \tau\cup \sigma \in S_K \}$. It is easy to verify that if $P$ is a poset and $\sigma$ is a simplex of $\K(P)$ then $\St_{\K(P)}(\sigma)=\K(\Sts_P(\sigma))$.

\subsection*{Cutsets and crosscuts}

Let $P$ be a poset. Following \cite{bjorner1981homotopy}, we say that a subset $S\subseteq P$ is \emph{bounded} if it is either bounded below or bounded above, and that a subset $A\subseteq P$ is \emph{astral} if there exists $x\in P$ such that $A\subseteq \st{x}$. Note that a subset $A\subseteq P$ is astral if and only if $\st{A}\neq\varnothing$.

A subset $X\subseteq P$ is a \emph{cutset} of $P$ if for every finite chain $\sigma$ contained in $P$ there exists an element $a\in X$ such that $\sigma \cup \{a\}$ is a chain (or equivalently, $\sigma\subseteq \st{a}$). We say that a cutset $X$ is \emph{coherent} if every finite non-empty bounded subset of $X$ has either a meet or a join in $P$. A subset $X\subseteq P$ is a \emph{crosscut} of $P$ if it is an antichain of $P$ which is a coherent cutset. Note that if $X$ is a crosscut of $P$ and $A$ is a subset of $X$ then $A$ is bounded if and only if it is astral.

If $X$ is a cutset of $P$, the \emph{crosscut complex} $\Gmc(P,X)$ is the simplicial complex whose simplices are the non-empty finite astral subsets of $X$ \cite[p.94]{bjorner1981homotopy}. Note that this definition might be extended to any subset $X\subseteq P$.

The following result is contained in the proof of Theorem 2.3 of \cite{bjorner1981homotopy}. We include it here for future reference and we give a proof, following \cite{bjorner1981homotopy}, for completeness. 

\begin{lemma}[Bj\"orner] \label{lemma_astral_bjorner}
Let $P$ be a poset and let $X$ be a coherent cutset of $P$. Let $\sigma\subseteq X$ be a non-empty finite astral subset of $P$. Then there exists $z\in \st{\sigma}$ such that $\st{\sigma}\subseteq \st{z}$. In particular, $\st{\sigma}=\Sts_{\st{\sigma}}(z)$. 
\end{lemma}

\begin{proof}
Since $\sigma$ is an astral subset there exists $y\in P$ such that $\sigma\subseteq \st{y}$. Let $\sigma_1=\sigma\cap P_{\leq y}$ and $\sigma_2=\sigma\cap P_{\geq y}$. Note that $\sigma=\sigma_1\cup\sigma_2$ and that $\sigma_1$ and $\sigma_2$ are bounded subsets of $X$. Since $\sigma\neq\varnothing$, without loss of generality we may assume that $\sigma_1\neq\varnothing$. And since $X$ is a coherent cutset, the subset $\sigma_1$ has either a meet or a join in $P$.

Suppose first that $\sigma_1$ has a meet in $P$ and let $z$ be that meet. We will prove that $\st{\sigma_1}\subseteq \st{z}$. Let $x\in \st{\sigma_1}$. If there exists $a\in\sigma_1$ such that $x\geq a$ then $x\geq z$ and thus $x\in \st{z}$. Otherwise, for all $a\in\sigma_1$ we have that $x\leq a$ and hence $x\leq z$ and thus $x\in \st{z}$. Therefore $\st{\sigma_1}\subseteq \st{z}$. Hence, $\st{\sigma}\subseteq \st{z}$. Note that $z\in \st{\sigma}$ since for all $a\in\sigma_1$ and for all $a'\in\sigma_2$ we have that $z\leq a \leq y \leq a'$.

Suppose now that $\sigma_1$ has a join in $P$ and let $z$ be that join. As in the previous case we will prove that $\st{\sigma_1}\subseteq \st{z}$. Let $x\in \st{\sigma_1}$. If there exists $a\in\sigma_1$ such that $x\leq a$ then $x\leq z$ and thus $x\in \st{z}$. Otherwise, for all $a\in\sigma_1$ we have that $x\geq a$ and hence $x\geq z$ and thus $x\in \st{z}$. Therefore $\st{\sigma_1}\subseteq \st{z}$. It follows that $\st{\sigma}\subseteq \st{z}$. Note that $z\in \st{\sigma}$ since for all $a\in\sigma_1$ and for all $a'\in\sigma_2$ we have that $a\leq z \leq y \leq a'$.
\end{proof}

\subsection*{Alexandroff spaces}

Alexandroff spaces are topological spaces which satisfy that any intersection of open subsets is an open subset. Alexandroff proved in \cite{alexandroff1937diskrete} that there is a bijective correspondence between preorder relations on a set $X$ and Alexandroff topologies on $X$, under which down-sets correspond to open sets. Moreover, this correspondence induces a bijection between partial orders and Alexandroff topologies which satisfy the T$_0$ separation axiom. In addition, under the previous correspondence, order-preserving maps between preordered sets correspond to continuous maps between the associated topological spaces. 

Recall also that Alexandroff spaces are locally path-connected and thus the connected components of an Alexandroff space coincide with its path-connected components. Moreover, the connected components of an Alexandroff space coincide with the connected components of the preordered set associated to it.

Following Alexandroff's correspondence, from now on we will regard any poset as an Alexandroff T$_0$--space and any order-preserving map between posets as a continuous map without further notice.

Let $P$ and $Q$ be posets and let $f,g\colon P\to Q$ be order-preserving maps such that $f\leq g$. It is not difficult to verify that the map $H\colon P\times [0,1] \to Q$ defined by
\begin{displaymath}
H(x,t) = \left\{\begin{array}{cl} f(x) & \textnormal{if $t<1$,} \\ g(x) & \textnormal{if $t=1$,} \end{array}
 \right.
\end{displaymath}
is continuous, and hence a homotopy between $f$ and $g$ relative to the set $\{x\in P \mid f(x)=g(x)\}$. As a consequence, if $f,g\colon P\to Q$ are order-preserving maps such that, for each $x\in P$, $f(x)$ and $g(x)$ are comparable, then 
$f$ and $g$ are homotopic maps since both of them are comparable to the order-preserving map $h=\max\{f,g\}$. Hence, if $P$ is a poset which satisfies that there exists $a\in P$ such that $P=\Sts_P(a)$ then $P$ is contractible since the identity map of $P$ is homotopic to the constant map with value $a$. In particular, every poset with a maximum or a minimum element is contractible.

In \cite{mccord1966singular}, McCord proved that if $P$ is a poset then there exists a weak homotopy equivalence $\mu_P\colon|\K(P)|\to P$, that is, a continuous map which induces isomorphisms on all the homotopy groups for any choice of the basepoint (and hence, it also induces isomorphisms on all the homology and cohomology groups). Moreover, he proves that if $f\colon P\to Q$ is an order-preserving map between posets then there is a commutative diagram
\begin{center}
\begin{tikzpicture}[x=2.7cm,y=2.2cm]
\draw (0,1) node (KP) {$|\K(P)|$};
\draw (1,1) node (KQ) {$|\K(Q)|$};
\draw (0,0) node (P) {$P$};
\draw (1,0) node (Q) {$Q$};

\draw[->] (P)--(Q) node [midway,below] {$f$};
\draw[->] (KP)--(KQ) node [midway,above] {$|\K(f)|$};
\draw[->] (KP)--(P) node [midway,left] {$\mu_P$};
\draw[->] (KQ)--(Q) node [midway,right] {$\mu_Q$};
\end{tikzpicture}
\end{center}
It follows that if the map $f$ is a weak homotopy equivalence then $|\K(f)|$ is a homotopy equivalence. In addition, in the same article, McCord also proved that if $K$ is a simplicial complex then there exists a weak homotopy equivalence $|K|\to \X(K)$.

\smallskip

Let $X$ and $Y$ be topological spaces. We say that $X$ and $Y$ are \emph{weak homotopy equivalent} if there exist $n\in \N$ and topological spaces $X_0,X_1,\ldots,X_n$ such that $X_0=X$, $X_n=Y$ and such that for each $j\in\{1,2,\ldots,n\}$ there exists either a weak homotopy equivalence $X_{j-1}\to X_j$ or a weak homotopy equivalence $X_j \to X_{j-1}$. Note that, if $P$ is a poset, then $P$ and $P^{\op}$ are weak homotopy equivalent by the results of McCord.

We say that a topological space $X$ is \emph{weakly contractible} if the only map from $X$ to the singleton is a weak homotopy equivalence, that is, if all the homotopy groups of $X$ are trivial. From the results of McCord it follows that a poset $P$ is weakly contractible if and only if $|\K(P)|$ is contractible. Recall also that a finite poset which is weakly contractible has the fixed point property (cf. \cite[Theorem 2.1]{baclawski1979fixed}).

\section{The crosscut poset}
\label{section_crosscut}

In this section we introduce the crosscut poset and show how it is related to the crosscut complex. By means of the crosscut poset we will obtain generalizations of results of \cite{baclawski1979fixed}, \cite{barmak2011quillen} and \cite{bjorner1981homotopy}, together with some applications to the fixed point property and the fixed simplex property.

\begin{notat}
Let $S$ be a set. The power set of $S$ will be denoted by $\mathcal{P}(S)$ and the set $\mathcal{P}(S)-\{\varnothing \}$ will be denoted by $\mathcal{P}_{\neq\varnothing}(S)$. In addition, the set whose elements are the finite non-empty subsets of $S$ will be denoted by $\mathcal{P}_{f,\neq \varnothing}(S)$.
\end{notat}

\begin{definition}
Let $P$ be a poset and let $X$ be a subset of $P$. We define the \emph{crosscut poset of $P$ with respect to $X$} as the subposet of $(\mathcal{P}_{\neq\varnothing}(P),\subseteq)$ whose elements are the connected components of the non-empty subposets $\st{A}$ with $A\in \mathcal{P}_{\neq\varnothing}(X)$. It will be denoted by $\Gm(P,X)$.
\end{definition}

\begin{ex}
Let $P$ be the poset defined by the following Hasse diagram
\begin{center}
\begin{tikzpicture}
\tikzstyle{every node}=[font=\scriptsize]
\foreach \x in {0,1} \draw (\x,0) node(\x){$\bullet$} node[below=1]{\x}; 
\draw (0,1) node (2){$\bullet$} node[left=1]{2};
\draw (1,1) node (3){$\bullet$} node[right=1]{3};
\draw (2,1) node (4){$\bullet$} node[right=1]{4};
\foreach \x in {5,6,7} \draw (\x-5,2) node(\x){$\bullet$} node[above=1]{\x}; 
  
\foreach \x in {2,3} \draw (0)--(\x);
\foreach \x in {2,3} \draw (1)--(\x);
\foreach \x in {5,6,7} \draw (2)--(\x);
\foreach \x in {5,6,7} \draw (3)--(\x);
\foreach \x in {6,7} \draw (4)--(\x);
\end{tikzpicture}
\end{center}
Note that $\st{\{5,6\}}=\st{\{5,7\}}=\st{\{5,6,7\}}=\{0,1,2,3\}$, which is connected, and that the connected components of $\st{\{6,7\}}$ are $\{0,1,2,3\}$ and $\{4\}$. Thus, the Hasse diagram of the poset $\Gm(P,\mxl(P))$ is
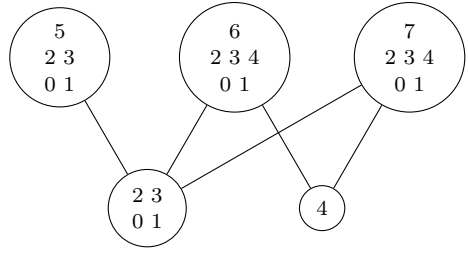
\begin{center}
\begin{tikzpicture}[x=2.3cm, y=2cm]
\draw node (U5) at (0,1) [CC] {5\\2 3 \\0 1};
\draw node (U6) at (1,1) [CC] {6\\2 3 4\\ 0 1};
\draw node (U7) at (2,1) [CC] {7\\2 3 4\\ 0 1};
\draw node (U56) at (0.5,0) [CC] {2 3\\0 1};
\draw node (U672) at (1.5,0) [CC] {4};
 
\draw [-] (U56) to (U5); 
\draw [-] (U56) to (U6);
\draw [-] (U56) to (U7);
\draw [-] (U672) to (U6);
\draw [-] (U672) to (U7);
\end{tikzpicture}
\end{center}
On the other hand, since $\{5,6,7\}$ is an astral subset of $P$ it follows that $\Gmc(P,\mxl(P))$ is a $2$--simplex. 

Observe also that $\Gmc(P,\mxl(P))=\Gmc(P-\{4\},\mxl(P))$ but $\Gm(P-\{4\},\mxl(P))$ is not isomorphic to $\Gm(P,\mxl(P))$ since  the Hasse diagram of the poset $\Gm(P-\{4\},\mxl(P))$ is
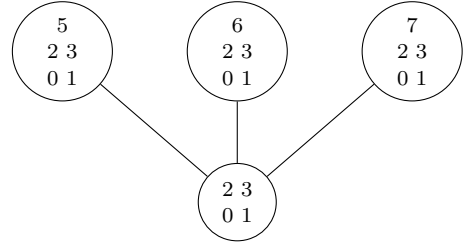
\begin{center}
\begin{tikzpicture}[x=2.3cm, y=2cm]
\draw node (U5) at (0,1) [CC] {5\\2 3 \\0 1};
\draw node (U6) at (1,1) [CC] {6\\2 3 \\ 0 1};
\draw node (U7) at (2,1) [CC] {7\\2 3 \\ 0 1};
\draw node (U56) at (1,0) [CC] {2 3\\0 1};
 
\draw [-] (U56) to (U5); 
\draw [-] (U56) to (U6);
\draw [-] (U56) to (U7);
\end{tikzpicture}
\end{center}
\end{ex}

Judging by the previous example, it seems that the crosscut poset is a finer invariant than the crosscut complex.
We will prove that this is indeed true under mild assumptions. To this end we will need the following lemma, which gives a characterization of the maximal elements of the crosscut poset.

\begin{lemma}
\label{lemma_mxl_of_Gm}
Let $P$ be a poset and let $X$ be an antichain of $P$. Then $\mxl(\Gm(P,X)) = \{\st{a} \mid a\in X \}$.
\end{lemma}

\begin{proof}
Let $a\in X$. Since $\st{a}$ is connected, $\st{a}\in \Gm(P,X)$. Let $C\in \Gm(P,X)$ such that $C\supseteq \st{a}$ and let $A$ be a non-empty subset of $X$ such that $C$ is a connected component of $\st{A}$. Then $\st{a}\subseteq C\subseteq \st{A}$. It follows that $a\in \st{x}$ for all $x\in A$. Since $X$ is an antichain we obtain that $a = x$ for all $x\in A$. Hence $A=\{a\}$ and $\st{a} = \st{A} = C$. Therefore, $\st{a}\in \mxl(\Gm(P,X))$.

Now let $C\in \mxl(\Gm(P,X))$. Let $A$ be a non-empty subset of $X$ such that $C$ is a connected component of $\st{A}$. Let $a\in A$. Then $C\subseteq \st{a}$. Hence $C=\st{a}$.
\end{proof}

\begin{prop} \label{prop_finer_invariant}
Let $P$ and $P'$ be posets, let $X$ be an antichain of $P$ and let $X'$ be an antichain of $P'$. If $\Gm(P,X)$ is isomorphic to $\Gm(P',X')$ then $\Gmc(P,X)$ is isomorphic to $\Gmc(P',X')$.
\end{prop}

\begin{proof}
Let $\varphi\colon \Gm(P,X) \to \Gm(P',X')$ be an isomorphism. Let $\iota\colon X \to \Gm(P,X)$ and $\iota'\colon X' \to \Gm(P',X')$ be given by $\iota(x)=\st{x}$ for all $x\in X$ and $\iota'(x')=\st{x'}$ for all $x'\in X'$. Note that $\iota$ and $\iota'$ are injective maps since $X$ and $X'$ are antichains. From Lemma \ref{lemma_mxl_of_Gm} it follows that $\iota$ and $\iota'$ induce bijections $X\cong \mxl(\Gm(P,X))$ and $X'\cong \mxl(\Gm(P',X'))$, respectively. Since $\varphi$ is an isomorphism, $\varphi$ induces a bijection  $\mxl(\Gm(P,X))\cong \mxl(\Gm(P',X'))$. Therefore, there exists a bijective map $\widetilde\varphi\colon X\to X'$ such that $\iota'(\widetilde\varphi(x))=\varphi(\iota(x))$ for all $x\in X$.

We will prove now that $\widetilde\varphi$ defines a morphism of simplicial complexes $\Gmc(P,X) \to \Gmc(P',X')$. Let $\sigma$ be a simplex of $\Gmc(P,X)$, that is, a non-empty finite astral subset of $X$. Then $\st{\sigma}\neq\varnothing$. Let $C$ be a connected component of $\st{\sigma}$. Then $C\subseteq \st{a}$ for all $a\in\sigma$. Hence, for each $a\in \sigma$ we have that
\begin{align*}
\varphi(C) &\subseteq \varphi(\st{a}) = \varphi(\iota(a)) = \iota'(\widetilde\varphi(a)) = \st{\widetilde\varphi(a)}.
\end{align*}
Thus, $\varphi(C)\subseteq \st{\widetilde\varphi(\sigma)}$ and hence $\st{\widetilde\varphi(\sigma)}\neq\varnothing$ which implies that $\widetilde\varphi(\sigma)$ is an astral subset of $X'$. Therefore, $\widetilde\varphi$ defines a morphism of simplicial complexes $\Gmc(P,X) \to \Gmc(P',X')$.

Applying the previous argument to $\varphi^{-1}$ we obtain a morphism of simplicial complexes $\Gmc(P',X') \to \Gmc(P,X)$, which is the inverse of $\widetilde\varphi$.
\end{proof}

We mention that the previous proposition does not hold if either $X$ or $X'$ is not an antichain. Similarly, Lemma \ref{lemma_mxl_of_Gm} might not hold if $X$ is not an antichain. In Example \ref{ex_counterexamples} we provide simple counterexamples.

\begin{ex} \label{ex_counterexamples}
Let $P$ be the poset with elements $0$ and $1$ such that $0<1$ and let $P'$ be a one-element poset. Then $\Gm(P,P)$ and $\Gm(P',P')$ are one-element posets but $\Gmc(P,P)$ is a $1$--simplex and $\Gmc(P',P')$ is a simplicial complex with a unique vertex.

On the other hand, let $Q$ be the poset with elements $a$, $b$ and $c$ with non-trivial relations $a<b$ and $a<c$. Let $X=\{a,b\}$. Then $\st{b}=\{a,b\}$ is not a maximal element of $\Gamma(Q,X)$ since it is properly contained in $\st{a}=Q$.
\end{ex}

\begin{definition}
Let $P$ be a poset and let $X$ and $B$ be subsets of $P$. We define 
\begin{displaymath}
\I_X(B)=\{x\in X \mid B\subseteq \st{x} \}.
\end{displaymath}
\end{definition}

\begin{rem} \label{rem_astral}
Let $P$ be a poset, let $X$ be a subset of $P$ and let $D$ be a non-empty subset of $P$. For each $d\in D$ and for each $x\in \I_X(D)$ we have that $d\in \st{x}$, or equivalently, $x\in \st{d}$. Therefore, $\I_X(D)\subseteq \st{D}$. In particular, every subset of $\I_X(D)$ is an astral subset of $P$.
\end{rem}

The following proposition states a key property of the crosscut poset.

\begin{prop} \label{prop_connected_subset}
Let $P$ be a poset and let $X$ be a subset of $P$. Let $B\subseteq P$ be a connected non-empty subposet and let $\Gamma_B=\{C\in \Gm(P,X) \mid B\subseteq C \}$. If $\Gamma_B$ is non-empty, then $\I_X(B)\neq \varnothing$ and $\Gamma_B$ has a minimum element, which is the connected component of $\st{\I_X(B)}$ that contains $B$.

In particular, if $B\in\Gm(P,X)$ then $\I_X(B)\neq \varnothing$ and $B$ is a connected component of $\st{\I_X(B)}$.
\end{prop}

\begin{proof}
Let $C_1\in\Gm(P,X)$ such that $B\subseteq C_1$ and let $A$ be a non-empty subset of $X$ such that $C_1$ is a connected component of $\st{A}$. Let $a_0\in A$. Then $B\subseteq \st{A}\subseteq \st{a_0}$. Hence $a_0\in\I_X(B)$. Thus, $\I_X(B)\neq\varnothing$. 

Clearly $B\subseteq \st{\I_X(B)}$ and, since $B$ is a connected non-empty subposet of $P$, there exists a unique connected component $C_0$ of $\st{\I_X(B)}$ such that $B\subseteq C_0$. Note that $C_0\in \Gm(P,X)$.

We will prove now that $C_0=\min\{C\in \Gm(P,X) \mid B\subseteq C \}$. Let $C\in \Gm(P,X)$ such that $B\subseteq C$. There exists a non-empty subset $A\subseteq X$ such that $C$ is a connected component of $\st{A}$. Hence $B\subseteq \st{a}$ for all $a\in A$. Then, $A\subseteq \I_X(B)$ and thus $\st{\I_X(B)}\subseteq \st{A}$. Hence $C_0\subseteq \st{A}$. Since $C$ is a connected component of $\st{A}$ and $\varnothing\neq B\subseteq C_0 \cap C$, it follows that $C_0\subseteq C$.
\end{proof}

We will now show how the crosscut poset is related to the crosscut complex. To this end, we need the following lemma.

\begin{lemma} \label{lemma_power_set_weakly_contractible}
Let $B$ be a non-empty set. Then the poset $(\mathcal{P}_{f,\neq \varnothing}(B),\subseteq)$ is weakly contractible.
\end{lemma}

\begin{proof}
Let $\Pc$ denote the poset $(\mathcal{P}_{f,\neq \varnothing}(B),\subseteq)$. Note that $\Pc$ is connected since, if $A_1,A_2\in \Pc$ then $A_1\cup A_2 \in \Pc$.

Now, let $n\in\N$, let $S^n$ denote the $n$--sphere and let $f\colon S^n \to \Pc$ be a continuous map. Since $\im f$ is compact and $\Pc$ is locally finite then $\im f$ is a finite subset of $\Pc$. Let $M=\bigcup\limits_{A\in \im f}A$. Note that $M\in \Pc$ and that $\im f\subseteq \Pc_{\leq M}$. Hence the continuous map $f$ can be factorized through the contractible poset $\Pc_{\leq M}$. Thus $f$ is nullhomotopic.

Therefore, $\Pc$ is weakly contractible.
\end{proof}

The following proposition exhibits the relationship between the crosscut poset and the crosscut complex.

\begin{prop} \label{prop_crosscut_poset_and_complex}
Let $P$ be a poset and let $X$ be a coherent cutset of $P$. Then $\X(\Gmc(P,X))$ and $\Gm(P,X)$ are weak homotopy equivalent. In particular, $|\K(\Gm(P,X))|$ and $|\Gmc(P,X)|$ are homotopy equivalent.
\end{prop}

\begin{proof}
Let $\nu\colon \X(\Gmc(P,X)) \to \Gm(P,X)^{\op}$ be defined by $\nu(\sigma)=\st{\sigma}$. Note that, if $\sigma\subseteq X$ is a non-empty finite astral subset of $P$ then $\st{\sigma}$ is non-empty and, by Lemma \ref{lemma_astral_bjorner}, connected. Thus, $\nu$ is well-defined. Clearly, $\nu$ is order-preserving.

We will prove now that $\nu$ is a weak homotopy equivalence. Note that for each $D\in \Gm(P,X)^{\op}$ we have that
\begin{align*}
\nu^{-1}(\Gm(P,X)^{\op}_{\leq D}) &= \{\sigma \in \X(\Gmc(P,X)) \mid \st{\sigma} \supseteq D \} = \mathcal{P}_{f,\neq \varnothing}(\I_X(D)),
\end{align*}
since if $\sigma$ is a non-empty subset of $X$ such that $\st{\sigma} \supseteq D$ then $\sigma \subseteq \I_X(D)$, and if $\sigma \in \mathcal{P}_{f,\neq \varnothing}(\I_X(D))$ then $\sigma$ is astral by Remark \ref{rem_astral} and $\st{\sigma}\supseteq \st{\I_X(D)} \supseteq D$. Thus, for each $D\in \Gm(P,X)^{\op}$, $\nu^{-1}(\Gm(P,X)^{\op}_{\leq D})$ is weakly contractible by Lemma \ref{lemma_power_set_weakly_contractible}, since $\I_X(D)\neq\varnothing$ by Proposition \ref{prop_connected_subset}. 

Also, note that $\{\Gm(P,X)^{\op}_{\leq D} \mid D\in \Gm(P,X)^{\op}\}$ is a basis-like\footnote{An open cover $\mathcal{U}$ of a topological space $Z$ is called \emph{basis-like} if for all $U,V\in \mathcal{U}$ and for all $z\in U\cap V$, there exists $W\in \mathcal{U}$ such that $z\in W\subseteq U\cap V$.} open cover of $\Gm(P,X)^{\op}$ (it is indeed a basis for the topology of $\Gm(P,X)^{\op}$). Therefore, $\nu$ is a weak homotopy equivalence by \cite[Theorem 6]{mccord1966singular}. The result follows.
\end{proof}

The following proposition sharpens Proposition \ref{prop_crosscut_poset_and_complex} in the case that the cutset $X$ is finite. 

\begin{prop} \label{prop_strong_deformation_retract}
Let $P$ be a poset and let $X$ be a coherent finite cutset of $P$. Then $\Gm(P,X)^{\op}$ is isomorphic to a strong deformation retract of $\X(\Gmc(P,X))$.
\end{prop}

\begin{proof}
Let $\nu\colon \X(\Gmc(P,X)) \to \Gm(P,X)^{\op}$ be the map defined in the previous proof. Let $\iota\colon \Gm(P,X)^{\op} \to \X(\Gmc(P,X))$ be defined by $\iota(C)=\I_X(C)$. Note that, for each $C\in \Gm(P,X)$, the set $\I_X(C)$ is non-empty by Proposition \ref{prop_connected_subset} and an astral subset of $P$ by Remark \ref{rem_astral}. Thus, $\iota$ is well defined. Clearly, $\iota$ is order-preserving.

Now, let $C\in\Gm(P,X)$. Since $\st{\I_X(C)}=\nu(\iota(C))\in \Gm(P,X)$ it follows that $\st{\I_X(C)}$ is connected and hence $C=\st{\I_X(C)}$ by Proposition \ref{prop_connected_subset}. Therefore, $\nu\iota=\id_{\Gm(P,X)^{\op}}$. On the other hand, let $\sigma$ be a non-empty finite astral subset of $X$. Clearly $\sigma\subseteq \I_X(\st{\sigma})$. Therefore, $\iota\nu\geq \id_{\X(\Gmc(P,X))}$. The result follows.
\end{proof}

The following theorem is one of the main results of this article.

\begin{theo} \label{theo_cutsets}
Let $P$ be a poset and let $X$ be a cutset of $P$. Let $\Gm_f(P,X)$ be the subposet of $\Gm(P,X)$ whose elements are the connected components of the non-empty sets $\st{A}$ with $A\in \mathcal{P}_{f,\neq\varnothing}(X)$. Suppose that for all $C\in \Gm_f(P,X)$, $C$ is weakly contractible. Then $P$ and $\Gm(P,X)$ are weak homotopy equivalent. In addition, if $P$ is a finite poset then $\K(P)$ and $\K(\Gm(P,X))$ are simple homotopy equivalent.
\end{theo}

\begin{proof}
Let $\sigma$ be a non-empty finite chain of $P$. Since $X$ is a cutset, $\I_X(\sigma)\neq \varnothing$. Clearly, $\sigma\subseteq \st{\I_X(\sigma)}$, and since $\sigma$ is connected, there exists a unique connected component $C_\sigma$ of $\st{\I_X(\sigma)}$ such that $\sigma\subseteq C_\sigma$.

Let $\mu\colon \X(\K(P))\to \Gm(P,X)$ be defined by $\mu(\sigma)=C_\sigma$ for all $\sigma\in \X(\K(P))$. We will prove that $\mu$ is an order-preserving map. Let $\sigma,\tau\in \X(\K(P))$ be such that $\sigma\subseteq \tau$. Then $\sigma\subseteq C_\tau$. Thus, $C_\sigma\subseteq C_\tau$ by Proposition \ref{prop_connected_subset}. Therefore, $\mu$ is an order-preserving map.

Let $\mathcal{U}=\{\Gm(P,X)_{\leq D} \mid D\in \Gm_f(P,X)\}$. We will prove that $\mathcal{U}$ is a basis-like open cover of $\Gm(P,X)$. Let $C\in \Gm(P,X)$. By Proposition \ref{prop_connected_subset}, $\I_X(C)\neq\varnothing$. Let $a\in \I_X(C)$. Then $C\subseteq\st{a}$ and $\st{a}\in \Gm_f(P,X)$. Thus $\mathcal{U}$ is a cover of $\Gm(P,X)$. Now, let $V_1,V_2\in \mathcal{U}$ and let $C_0\in V_1\cap V_2$. Let $D_1,D_2\in \Gm_f(P,X)$ such that $V_1=\Gm(P,X)_{\leq D_1}$ and $V_2=\Gm(P,X)_{\leq D_2}$. Since $D_1,D_2\in \Gm_f(P,X)$, there exist finite non-empty subsets $A_1,A_2\subseteq X$ such that $D_1$ is a connected component of $\st{A_1}$ and $D_2$ is a connected component of $\st{A_2}$. Thus, 
\begin{displaymath}
C_0\subseteq D_1\cap D_2 \subseteq \st{A_1} \cap \st{A_2} = \st{A_1\cup A_2}.
\end{displaymath}
Since $C_0$ is connected, there exists a (unique) connected component $D_3$ of $\st{A_1\cup A_2}$ such that $C_0\subseteq D_3$. Clearly, $D_3\in \Gm_f(P,X)$. Note that $D_3\subseteq \st{A_1}$ and since $D_3$ is connected and $D_3\cap D_1 \supseteq C_0 \neq \varnothing$ we obtain that $D_3\subseteq D_1$. Similarly, $D_3\subseteq D_2$. Thus, $C_0\in \Gm(P,X)_{\leq D_3} \subseteq V_1\cap V_2$. Therefore, $\mathcal{U}$ is a basis-like open cover of $\Gm(P,X)$.

We will prove now that, for all $C\in \Gm(P,X)$, 
\begin{displaymath}
\mu^{-1}(\Gm(P,X)_{\leq C})=\{\sigma\in \X(\K(P)) \mid \sigma\subseteq C \}. 
\end{displaymath}
Let $C\in \Gm(P,X)$. If $\sigma\in \X(\K(P))$ is such that $\mu(\sigma) \in \Gm(P,X)_{\leq C}$, then $\sigma\subseteq C_\sigma \subseteq C$. Conversely, if $\sigma\in \X(\K(P))$ is such that $\sigma\subseteq C$ then $\mu(\sigma)=C_\sigma\subseteq C$ by Proposition \ref{prop_connected_subset}.

Since for all $C\in \Gm_f(P,X)$, the subposet $\{\sigma\in \X(\K(P)) \mid \sigma\subseteq C \}$ is isomorphic to $\X(\K(C))$, which is weakly contractible by hypothesis, it follows that $\mu$ is a weak homotopy equivalence by \cite[Theorem 6]{mccord1966singular}. The second assertion of the theorem follows from \cite[Theorem 1.2]{barmak2011quillen} and the fact that a finite simplicial complex has the same simple homotopy type as its barycentric subdivision.
\end{proof}

In \cite[Theorem 2.3]{bjorner1981homotopy}, Bj\"orner proves that if $P$ is a poset and $X\subseteq P$ is a coherent cutset then $|\K(P)|$ and $|\Gmc(P,X)|$ are homotopy equivalent. We will show that Theorem \ref{theo_cutsets} generalizes Bj\"orner's result. To this end, we will use the following proposition which exhibits the relationship between the coherent cutset condition and the elements of the crosscut poset.

\begin{prop} \label{prop_coherent_cutset_condition}
Let $P$ be a poset and let $X$ be a coherent cutset of $P$. Let $\Gm_f(P,X)$ be the subposet of $\Gm(P,X)$ defined in Theorem \ref{theo_cutsets}. Then:
\begin{enumerate}[(a)]
\item $\Gm_f(P,X)=\{\st{A} \mid A \textnormal{ is a finite non-empty astral subset of }X\}$.
\item For each $C\in \Gm_f(P,X)$, there exists $z\in C$ such that $C=\Sts_C(z)$. In particular, the elements of $\Gm_f(P,X)$ are contractible subspaces of $P$.
\end{enumerate}
\end{prop}

\begin{proof}
Note that if $A$ is a finite non-empty astral subset of $X$ then $\st{A}$ is connected by Lemma \ref{lemma_astral_bjorner} and thus $\st{A}\in \Gm_f(P,X)$. Conversely, let $C\in \Gm_f(P,X)$ and let $A$ be a non-empty finite subset of $X$ such that $C$ is a connected component of $\st{A}$. Note that $A$ is an astral subset of $P$ and thus, by Lemma \ref{lemma_astral_bjorner}, there exists $z\in \st{A}$ such that $\st{A}=\Sts_{\st{A}}(z)$. In particular, $\st{A}$ is connected and thus $C=\st{A}$. The result follows.
\end{proof}

Now, let $P$ be a poset and let $X$ be a coherent cutset of $P$. By Proposition \ref{prop_coherent_cutset_condition}, the hypotheses of Theorem \ref{theo_cutsets} hold, and thus $|\K(P)|$ and $|\K(\Gm(P,X))|$ are homotopy equivalent. And since, by Proposition \ref{prop_crosscut_poset_and_complex}, $|\K(\Gm(P,X))|$ and $|\Gmc(P,X)|$ are homotopy equivalent, we obtain that $|\K(P)|$ and $|\Gmc(P,X)|$ are homotopy equivalent. Therefore, Theorem \ref{theo_cutsets} generalizes \cite[Theorem 2.3]{bjorner1981homotopy}. Note that the hypotheses of Theorem \ref{theo_cutsets} are much weaker than those of \cite[Theorem 2.3]{bjorner1981homotopy}.

\medskip

Observe also that Theorem \ref{theo_cutsets} generalizes \cite[Theorem 4.5]{barmak2011quillen} which states that if $P$ is a finite poset and $X\subseteq P$ is a crosscut then $\K(P)$ and $\Gmc(P,X)$ are simple homotopy equivalent. Indeed, under the hypotheses of \cite[Theorem 4.5]{barmak2011quillen} the finite spaces $\Gm(P,X)^{\op}$ and $\X(\Gmc(P,X))$ are homotopy equivalent by Proposition \ref{prop_strong_deformation_retract}, hence $\K(\Gm(P,X))$ and $\K(\X(\Gmc(P,X)))$ have the same simple homotopy type by \cite[Theorem 3.10]{barmak2008simple}, and thus $\K(\Gm(P,X))$ and $\Gmc(P,X)$ are simple homotopy equivalent.

\begin{ex} \label{ex_2_sect_3}
Let $P$ be the poset given by the following Hasse diagram.
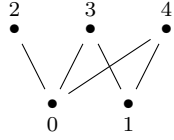
\begin{center}
\begin{tikzpicture}
\tikzstyle{every node}=[font=\scriptsize]
\foreach \x in {0,1} \draw (\x,0) node(\x){$\bullet$} node[below=1]{\x}; 
\foreach \x in {2,3,4} \draw (\x-2.5,1) node(\x){$\bullet$} node[above=1]{\x}; 
\foreach \x in {2,3,4} \draw (0)--(\x);
\foreach \x in {3,4} \draw (1)--(\x);
\end{tikzpicture}
\end{center}
The set of minimal elements of $P$ is a cutset of $P$ which is not coherent since the set $\{0,1\}$ is a bounded subset which has neither a meet nor a join. The poset $\Gm(P,\mnl(P))$ and the simplicial complex $\Gmc(P,\mnl(P))$ are shown below.
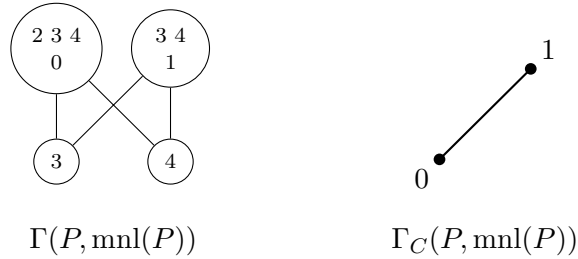
\begin{center}
\begin{tikzpicture}[x=1.5cm, y=1.5cm]
\draw (0.5,-0.7) node(CP){\normalsize $\Gm(P,\mnl(P))$};
\draw node (F01a) at (0,0) [CC] {3};
\draw node (F01b) at (1,0) [CC] {4};
\draw node (F0) at (0,1) [CC] {2 3 4\\0};
\draw node (F1) at (1,1) [CC] {3 4\\1};
\draw [-] (F0) to (F01a) to (F1) to (F01b) to (F0);
\end{tikzpicture}
\hspace*{2cm}
\begin{tikzpicture}[x=1.2cm, y=1.2cm]
\draw (0.5,-0.9) node(CP){\normalsize $\Gmc(P,\mnl(P))$};
\draw(0,0) node[below left]{$0$};
\fill[fill=black] (0,0) circle (.2em);
\draw(1,1) node[above right]{$1$};
\fill[fill=black] (1,1) circle (.2em);
\draw[thick] (0,0) -- (1,1);
\end{tikzpicture}
\end{center}
Clearly $|\K(P)|$ is homotopy equivalent to the $1$--sphere $S^1$ and $|\Gmc(P,\mnl(P))|$ is contractible, thus the assertions of \cite[Theorem 2.3]{bjorner1981homotopy} and \cite[Theorem 4.5]{barmak2011quillen} do not hold without the assumption that the cutset is coherent. Observe also that the poset $P$ does satisfy the hypotheses of Theorem \ref{theo_cutsets} and clearly $|\K(P)|$ and $|\K(\Gm(P,\mnl(P)))|$ are homotopy equivalent.
\end{ex}

We will now study the particular case in which the cutset is the subset of maximal elements and every element of the corresponding crosscut poset has a maximum element. Let $P$ be a poset. Observe that every $C\in \Gm(P,\mxl(P))$ is a down-set of $P$ since, for all $a\in\mxl(P)$, $\st{a}=P_{\leq a}$ is a down-set of $P$. Note also that $\mxl(P)$ is a cutset of $P$ if and only if $P=\bigcup\limits_{a\in\mxl(P)} P_{\leq a}$ and that if $P$ is a finite poset then $\mxl(P)$ is a cutset of $P$.

\begin{theo}\label{theo_every_C_has_maximum}
Let $P$ be a poset such that $\mxl(P)$ is a cutset of $P$. Suppose that every $C\in \Gm(P,\mxl(P))$ has a maximum element. Let $P_{0}=\{\max C \mid C\in \Gm(P,\mxl(P))\}$ and let $i\colon P_{0}\to P$ be the inclusion map. Then:
\begin{enumerate}[($a$)]
\item There exists a retraction $r\colon P\to P_{0}$ such that $ir\geq \id_P$. In particular, $P_{0}$ is a strong deformation retract of $P$.
\item The posets $\Gm(P,\mxl(P))$ and $P_{0}$ are isomorphic.
\end{enumerate}
\end{theo}

\begin{proof} \ 

($a$) For each $x\in P$ let $C_x$ be the connected component of $\st{\I_{\mxl(P)}({\{x\}})}$ that contains $x$. Observe that $\I_{\mxl(P)}({\{x\})}\neq \varnothing$ since $\mxl(P)$ is a cutset of $P$. Clearly, $C_x\in\Gm(P,\mxl(P))$.

Let $r\colon P \to P_{0}$ be defined by $r(x)=\max C_x$. We claim that $r$ is an order-preserving map. Indeed, let $x_1,x_2\in P$ such that $x_1\leq x_2$. Since $C_{x_2}$ is a down-set of $P$ we obtain that $x_1\in C_{x_2}$. And since, by Proposition \ref{prop_connected_subset}, $C_{x_1}$ is the minimum element of $\{C\in \Gm(P,\mxl(P)) \mid \{x_1\}\subseteq C \}$, we obtain that $C_{x_1}\subseteq C_{x_2}$. Hence, $r(x_1)\leq r(x_2)$.

Clearly, $ir\geq \id_P$. We will prove now that $ri=\id_{P_{0}}$. Let $z\in P_{0}$. Then there exists $C'\in \Gm(P,\mxl(P))$ such that $z=\max C'$. By Proposition \ref{prop_connected_subset}, $C_{z}$ is the minimum element of $\{C\in \Gm(P,\mxl(P)) \mid \{z\}\subseteq C \}$. Hence, $C_z\subseteq C'$. Thus, $\max C_z = z$ and hence $ri(z)=z$. 

($b$) Let $\varphi\colon \Gm(P,\mxl(P)) \to P_{0}$ and $\psi \colon P_{0} \to \Gm(P,\mxl(P))$ be given by $\varphi(C)=\max C$ and $\psi(z)=P_{\leq z}$. We will prove that $\psi$ is well-defined and that $\varphi$ and $\psi$ are mutually inverse isomorphisms.

Let $z\in P_{0}$ and let $C\in \Gm(P,\mxl(P))$ such that $z=\max C$. Since $C$ is a down-set of $P$ we obtain that $C=P_{\leq z}$. Therefore, $\psi$ is well-defined and $C=\psi(\varphi(C))$. It follows that $\psi\varphi=\id_{\Gm(P,\mxl(P))}$. Clearly, $\varphi\psi=\id_{P_{0}}$ and $\varphi$ and $\psi$ are order-preserving maps.
\end{proof}

Now, we will give some applications of the previous theorem. Firstly, let $P$ be a finite poset such that $\mnl(P)$ is a coherent cutset of $P$. Let $P_M$ be the subposet of $P$ consisting of all the elements that can be obtained as joins of non-empty subsets of $\mnl(P)$ and let $\widehat{P}_M$ be the lattice obtained by adding a maximum and a minimum element to $P_M$. Under these assumptions, \cite[Proposition 3.1]{baclawski1979fixed} states that if the lattice $\widehat{P}_M$ is non-complemented then the poset $P$ has the strong fixed point property. However, in the proof of this result Baclawski and Bj\"orner obtain something stronger: that the poset $P$ is weakly contractible. In addition, by the proof of \cite[Corollary 3.2]{baclawski1979fixed} one obtains that if the lattice $\widehat{P}_M$ is non-complemented then $P_M$ is weakly contractible. Thus, with the previous notations, we can formulate a stronger version of \cite[Proposition 3.1]{baclawski1979fixed} as follows.

\begin{prop} \label{prop_Baclawski_Bjorner_generalized}
Let $P$ be a finite poset such that $\mnl(P)$ is a coherent cutset of $P$. If $P_M$ is weakly contractible then $P$ is weakly contractible.
\end{prop}

\begin{proof}
We will prove first that every $C\in \Gm(P,\mnl(P))$ has a minimum element and that $P_M=\{\min C \mid C\in \Gm(P,\mnl(P))\}$.
Let $A\in \mathcal{P}_{\neq\varnothing}(\mnl(P))$ such that $\st{A}\neq\varnothing$. Since $\mnl(P)$ is a coherent cutset we obtain that the subset $A$ has a join. Hence $\st{A}$ has a minimum element and, in particular, $\st{A}$ is connected. It follows that 
\begin{displaymath}
\Gm(P,\mnl(P))=\{\st{A} \mid A\in \mathcal{P}_{\neq\varnothing}(\mnl(P)) \textnormal{ such that }\st{A}\neq\varnothing\}\end{displaymath}
and that every $C\in \Gm(P,\mnl(P))$ has a minimum element. It is not difficult to verify that $P_M=\{\min C \mid C\in \Gm(P,\mnl(P))\}$.

Therefore, $P$ and $P_M$ are homotopy equivalent by item ($a$) of (the dual version of) Theorem \ref{theo_every_C_has_maximum}. The result follows.
\end{proof}

The previous proof shows that the subposet $P_0$ of the dual version of Theorem \ref{theo_every_C_has_maximum} coincides with the subposet $P_M$ defined in \cite{baclawski1979fixed} when the cutset $\mnl(P)$ is coherent. Observe also that the hypotheses of Theorem \ref{theo_every_C_has_maximum} are much weaker than those of Proposition \ref{prop_Baclawski_Bjorner_generalized}. 

\medskip

Now, we will apply Theorem \ref{theo_every_C_has_maximum} to obtain some results concerning the fixed point property. Firstly, observe that if $P$ is a finite poset such that every $C\in \Gm(P,\mxl(P))$ has a maximum element then, by Theorem \ref{theo_every_C_has_maximum}, $P$ and $\Gm(P,\mxl(P))$ are homotopy equivalent finite T$_0$--spaces and thus, by \cite[Corollary 10.1.4]{barmak2011algebraic}, $P$ has the fixed point property if and only if $\Gm(P,\mxl(P))$ has the fixed point property. The following proposition generalizes this result for a not necessarily finite poset $P$.

\begin{prop} \label{prop_P_weakly_contractible}
Let $P$ be a poset such that $\mxl(P)$ is a cutset of $P$. Suppose that every non-empty subset $A\subseteq P$ such that $A^{\op}$ is well-ordered has an infimum in $P$ and that every $C\in \Gm(P,\mxl(P))$ has a maximum element. Then the poset $P$ has the fixed point property if and only if the poset $\Gm(P,\mxl(P))$ has the fixed point property.
\end{prop}

\begin{proof}
Let $P_{0}=\{\max C \mid C\in \Gm(P,\mxl(P))\}$. We will prove that the poset $P$ has the fixed point property if and only if the poset $P_{0}$ has the fixed point property, and thus the result will follow from the second item of Theorem \ref{theo_every_C_has_maximum}.

Let $i\colon P_0\to P$ be the inclusion map. By the first item of Theorem \ref{theo_every_C_has_maximum}, there exists a retraction $r\colon P\to P_{0}$ such that $ir\geq \id_P$. Thus, if $P$ has the fixed point property then $P_{0}$ has the fixed point property since $P_{0}$ is a retract of $P$. Now, suppose that $P_{0}$ has the fixed point property and let $f\colon P\to P$ be an order-preserving map. Thus, the map $rfi\colon P_{0}\to P_{0}$ has a fixed point $a$. Then, $i(a)=irfi(a)\geq f(i(a))$. Hence, the map $f$ has a fixed point by the Abian-Brown theorem (\cite[Theorem 2]{abian1961theorem}). Therefore, the poset $P$ has the fixed point property.
\end{proof}

The following result is another application of Theorem \ref{theo_every_C_has_maximum} which is related to the fixed point property.

\begin{prop} \label{prop_fixed_simplex_property}
Let $K$ be a simplicial complex such that every simplex of $K$ is contained in a maximal simplex of $K$. Let $M$ be the set of maximal simplices of $K$ and let $L_K$ be the poset of non-empty intersections of simplices of $M$, ordered by inclusion. Then $K$ has the fixed simplex property if and only if the poset $L_K$ has the fixed point property.
\end{prop}

\begin{proof}
By \cite[Proposition 9.20]{schroder2016ordered}, $K$ has the fixed simplex property if and only if the poset $\X(K)$ has the fixed point property. Note that $M=\mxl(\X(K))$ is a cutset of $\X(K)$.

We claim that for each $A\in \mathcal{P}_{\neq\varnothing}(M)$ the subset $\st{A}$ of $\X(K)$ has a maximum element if it is not empty. Indeed, let $A\subseteq M$ be a non-empty subset such that $\st{A}\neq \varnothing$. Let $\sigma_A=\bigcap\limits_{s\in A}s$. Then $\sigma_A\neq\varnothing$ since $\st{A}\neq \varnothing$. Thus $\sigma_A\in \X(K)$ and clearly $\sigma_A=\max \st{A}$. 

In particular, note that $\Gm(\X(K),M)=\{\st{A}\mid A\in \mathcal{P}_{\neq\varnothing}(M) \textnormal{ such that } \st{A}\neq \varnothing\}$. It is not difficult to verify that $L_K=\{\max C \mid C\in \Gm(\X(K),M)\}$. Let $i\colon L_K\to \X(K)$ be the inclusion map. By Theorem \ref{theo_every_C_has_maximum}, there exists a retraction $r\colon \X(K)\to L_K$ such that $ir\geq \id_{\X(K)}$. Hence, if $\X(K)$ has the fixed point property then $L_K$ has the fixed point property. Now suppose that $L_K$ has the fixed point property and let $f\colon \X(K)\to \X(K)$ be an order-preserving map. As in the previous proof, the map $rfi\colon L_K\to L_K$ has a fixed point $\sigma$, and then $i(\sigma)=irfi(\sigma)\geq f(i(\sigma))$. Since $\X(K)$ is locally finite it follows that the map $f$ has a fixed point. Therefore, the poset $\X(K)$ has the fixed point property if and only if the poset $L_K$ has the fixed point property. The result follows.
\end{proof}

The previous proof shows that Proposition \ref{prop_fixed_simplex_property} could have been obtained as a consequence of Proposition \ref{prop_P_weakly_contractible} and item ($b$) of Theorem \ref{theo_every_C_has_maximum}. We preferred to keep the proof of Proposition \ref{prop_fixed_simplex_property} as it is to show that, in this case, the application of the Abian-Brown theorem can be replaced by a much simpler finiteness argument.

\bibliographystyle{acm}
\bibliography{ref_crosscut}

\end{document}